\def\version{16/01/2013 Version 2\hfill
\href{http://arxiv.org/abs/1207.4947}{arXiv:1207.4947}}

\documentclass[a4paper,11pt]{amsart}
\usepackage{fullpage}

\usepackage{amssymb,amscd,amsxtra}

\usepackage[all]{xy}
\newdir{ >}{{}*!/-8pt/@{>}}

\usepackage{hyperref}

\theoremstyle{plain}
\newtheorem{thm}{Theorem}[section]
\newtheorem{lem}[thm]{Lemma}
\newtheorem{prop}[thm]{Proposition}

\theoremstyle{definition}
\newtheorem{rem}[thm]{Remark}

\numberwithin{equation}{section}

\def\ie{\emph{i.e.}}

\def\:{\colon}
\def\.{\cdot}

\def\<{\left\langle}
\def\>{\right\rangle}
\def\({\left(}
\def\){\right)}

\def\epsilon{\varepsilon}
\def\phi{\varphi}
\def\leq{\leqslant}
\def\geq{\geqslant}

\def\lra{\longrightarrow}
\def\Lra{\Longrightarrow}
\def\ra{\rightarrow}

\def\tilde#1{\widetilde{#1}}
\def\iso{\cong}

\def\A{\mathbb{A}}
\def\C{\mathbb{C}}

\def\F{\mathbb{F}}

\def\H{\mathbb{H}}
\def\k{\Bbbk}

\def\Q{\mathbb{Q}}
\def\N{\mathbb{N}}
\def\R{\mathbb{R}}

\def\Z{\mathbb{Z}}

\def\P{\mathbb{P}}

\def\Times_#1{\mathop{\times}_{#1}}
\def\oTimes_#1{\mathop{\otimes}_{#1}}
\def\oPlus_#1{\mathop{\bigoplus}_{#1}}

\DeclareMathOperator{\Tor}{Tor}

\DeclareMathOperator*{\colim}{colim}

\def\CP{\C P}
\def\CPi{\CP^\infty}
\def\HP{\H P}
\def\HPi{\HP^\infty}
\def\RP{\R P}
\def\RPi{\RP^\infty}

\def\BP{\mathit{BP}}

\def\bp1{\BP\langle 1\rangle}

\def\nsym{\mathsf{NSymm}}
\def\lsc{\Omega \Sigma \CPi}

\def\op{\mathrm{op}}
\def\SS{\mathbb{S}}
\def\TAQ{\mathit{TAQ}}
\def\Sq{\mathit{Sq}}
\def\k{{\boldsymbol{k}}}
\def\Sc{S}

\begin{document}
\title[Some properties of $M\xi$]
{Some properties of the Thom spectrum over loop suspension
of complex projective space}
\author{Andrew Baker \and Birgit Richter}
\address{School of Mathematics \& Statistics, University
of Glasgow, Glasgow G12 8QW, Scotland.}
\email{a.baker@maths.gla.ac.uk}
\urladdr{http://www.maths.gla.ac.uk/$\sim$ajb}
\address{Fachbereich Mathematik der Universit\"at Hamburg,
20146 Hamburg, Germany.}
\email{richter@math.uni-hamburg.de}
\urladdr{http://www.math.uni-hamburg.de/home/richter/}
\thanks{We would like to thank Jack Morava and Nitu Kitchloo
for encouraging us to work out details for some of the loose
ends left over from our earlier paper. We also thank Tyler
Lawson for introducing the first named author to the spectrum
$\tilde{\P}(\Sigma^{\infty-2}\CPi)$ and its universal property.}
\date{\version}

\begin{abstract}
This note provides a reference for some properties of the Thom
spectrum $M\xi$ over $\Omega\Sigma\CPi$.  Some of this material
is used in recent work of Kitchloo and Morava. We determine the
$M\xi$-cohomology of $\CPi$  and show that $M\xi^*(\CPi)$ injects
into power series over the algebra of non-symmetric functions.
We show that $M\xi$ gives rise to a commutative formal group
law over the non-commutative ring $\pi_*M\xi$. We also discuss
how $M\xi$ and some real and quaternionic analogues behave with
respect to spectra that are related to these Thom spectra by
splittings and by maps.
\end{abstract}

\maketitle

\section*{Introduction}

The map $\CPi = BU(1) \ra BU$ gives rise to a canonical loop map
$\lsc \to BU$. Therefore the associated Thom spectrum has a strictly
associative multiplication. But as is visible from its homology, which
is a tensor algebra on the reduced homology of $\CPi$, it is not even
homotopy commutative. This homology ring coincides with the ring of
non-symmetric functions, $\nsym$. We show that there is a map from
the $M\xi$-cohomology of $\CPi$ to the power series over the ring
of non-symmetric functions, $\nsym$. This result is used in~\cite{MK}
in an application of $M\xi$ to quasitoric manifolds.

Although $M\xi$ maps to $MU$, there is no obvious map to it from $MU$,
so a priori it is unclear whether there is a formal group law associated
to $M\xi$. However, analogues of the classical Atiyah-Hirzebruch spectral
sequence calculations for $MU$ can be made for $M\xi$, and these show
that there is a `commutative formal group' structure related to
$M\xi^*\CPi$, even though the coefficient ring $M\xi_*$ is not
commutative and its elements do not commute with the variables coming
from the choices of complex orientations. A formal group law in this
context is an element $F(x,y) \in M\xi^*(\CPi\times \CPi)$ that
satisfies the usual axioms for a commutative formal group law. We
describe the precise algebraic structure arising here in
Section~\ref{sec:FGL/Mxi}.

For $MU$ the $p$-local splitting gives rise to a map of ring
spectra $BP \ra MU_{(p)}$. We show that despite the fact that
$M\xi_{(p)}$ splits into (suspended) copies of $BP$, there
is no map of ring spectra $BP \ra M\xi_{(p)}$. For the
canonical Thom spectrum over $\Omega \Sigma \RPi$, $M\xi_\R$,
we show that the map of $E_2$-algebra spectra $H\F_2 \ra MO$
does not give rise to a ring map $H\F_2\ra M\xi_\R$.

For a map of ring spectra $MU\ra E$ to some commutative $S$-algebra
$E$ one can ask whether a map of commutative $S$-algebras
$S \wedge_{\P(\Sc)}\P(\Sigma^{\infty-2}\CPi)\ra E$ factors over~$MU$.
Here $\P(-)$ denotes the free commutative $S$-algebra functor.
It is easy to see that $MU$ is not equivalent to
$S \wedge_{\P(\Sc)}\P(\Sigma^{\infty-2}\CPi)$. We show that
there are commutative $S$-algebras for which that is not the case.
In the associative setting, the analoguous universal gadget would
be $S \amalg_{\A(S)} \A(\Sigma^{\infty-2}\CPi)$, where $\A(-)$ is
the free associative $S$-algebra functor and $\amalg$ denotes the
coproduct in the category of associative $S$-algebras. It is obvious
that the homology of $S\amalg_{\A(S)}\A(\Sigma^{\infty-2}\CPi)$
is much bigger than the one of $M\xi$. If we replace the coproduct
by the smash product, there is still a canonical map
$S \wedge_{\A(S)} \A(\Sigma^{\infty-2}\CPi) \ra M\xi$ due to the
coequalizer property of the smash product. However, this smash
product still has homology that is larger than that of~$M\xi$.
Therefore the freeness of $\lsc$ is not reflected on the level
of Thom spectra.

\section{The Thom spectrum of $\xi$}\label{sec:ThomSpectrum-xi}

Lewis showed \cite[Theorem IX.7.1]{LMS} that an $n$-fold loop
map to $BF$ gives rise to an $E_n$-structure on the associated
Thom spectrum. Here $E_n$  is the product of the little $n$-cubes
operad with the linear isometries operad. For a more recent
account in the setting of symmetric spectra see work of Christian
Schlichtkrull~\cite{Sch}.

The map $j\:\lsc \ra BU$ is a loop map and so the Thom spectrum
$M\xi$ associated to that map is an $A_\infty$ ring spectrum and
the natural map $M\xi \ra MU$ is one of $A_\infty$ ring spectra,
or equivalently of $S$-algebras in the sense of~\cite{EKMM}. Since
the homology $H_*(M\xi)$ is isomorphic as a ring to $H_*(\lsc)$,
we see that $M\xi$ is not even homotopy commutative. We investigated
some of the properties of $M\xi$ in~\cite{BR:qsym}.

For any commutative ring $R$, under the Thom isomorphism
$H_*(\lsc;R)\iso H_*(M\xi;R)$, the generator $Z_i$ corresponds
to an element $z_i \in H_{2i}(M\xi;R)$, where we set $z_0=1$.
Thomifying the map $i\:\CPi \ra \lsc$, we obtain a map
$Mi\: MU(1) \ra \Sigma^2M\xi$, and it is easy to see
that
\begin{equation}\label{eqn:beta-z}
Mi_*\beta_{i+1}=z_i.
\end{equation}

\subsection{Classifying negatives of bundles}\label{sec:NegBdles}

For every based space $X$, time-reversal of loops is
a loop-map from $(\Omega X)^\op$ to $\Omega X$, \ie,
\begin{equation*}
\bar{(\; .\; )} \:(\Omega X)^\op \ra \Omega X;
\quad
w \mapsto \bar{w},
\end{equation*}
where $\bar{w}(t) = w(1-t)$. Here $(\Omega X)^\op$ is the
space of loops on $X$ with the opposite multiplication of
loops.

We consider $BU$ with the $H$-space structure coming from
the Whitney sum of vector bundles and denote this space by
$BU_\oplus$. A complex vector bundle of finite rank on a
reasonable space $Y$ is represented by a map $f\:Y\ra BU$
and the composition
\begin{equation*}
\xymatrix@1{
Y \ar[r]^f & BU_\oplus^\op \ar[r]^{\bar{(\;.\;)}} & BU_\oplus
}
\end{equation*}
classifies the negative of that bundle, switching the r\^ole
of stable normal bundles and stable tangent bundles for smooth
manifolds.

For line bundles $g_i\: Y \ra \CPi \ra BU_\oplus$ ($i=1,\ldots,n$)
we obtain a map
\begin{equation*}
g=(g_n,\ldots,g_1)\: Y \ra Y^n \ra(\CPi)^n
                \ra (\Omega\Sigma\CPi)^\op\ra BU_\oplus^\op
\end{equation*}
and the composition with loop reversal classifies the negative
of the sum $g_n \oplus \cdots \oplus g_1$ as indicated in the
following diagram.

\bigskip
\begin{equation*}
\xymatrix{
Y \ar[r] \ar@/_3pc/[drrr]_{\bar{g}}\ar@/^22pt/[rrr]^{g}
 & {Y^n} \ar[r]
 & {(\Omega\Sigma\CPi)^\op} \ar[r] \ar[d]^{\bar{(\;.\;)}}
 & {BU_\oplus^\op} \ar[d]^{\bar{(\;.\;)}}  \\
& & {\Omega\Sigma\CPi} \ar[r] & {BU_\oplus}
}
\end{equation*}

\bigskip
\noindent
In this way, the splitting of the stable tangent bundle
of a toric manifold into a sum of line bundles can be
classified by $\lsc$. For work on an interpretation of
$\pi_*M\xi$ as the habitat for cobordism classes of
quasitoric manifolds see~\cite{MK}.

\section{$M\xi$-(co)homology} \label{sec:mxistern}

We note that the composition of the natural map $i\:\CPi\ra\lsc$
with $j\: \lsc \ra BU$ classifies the reduced line bundle $\eta-1$
over $\CPi$. The associated map
$Mi\: \Sigma^\infty MU(1) \ra \Sigma^2 M\xi$ gives a distinguished
choice of complex orientation $x_\xi\in \widetilde{M\xi}^2(\CPi)$,
since the zero-section $\CPi\ra MU(1)$ is an equivalence.

We use the Atiyah-Hirzebruch spectral sequence
\begin{equation}\label{eqn:ahss}
\mathrm{E}_2^{*,*}= H^*(\CPi;M\xi^*) \Lra M\xi^*(\CPi).
\end{equation}
As $M\xi$ is an associative ring spectrum, this spectral
sequence is multiplicative and its $\mathrm{E}_2$-page
is $\Z[[x]]\otimes M\xi^*$. As the spectral sequence
collapses, the associated graded is of the same form
and we can deduce the following:
\begin{lem}\label{lem:MxiCPi}
As a left $M\xi^*=M\xi_{-*}$-module we have
\begin{equation}\label{eqn:MxiCPi-LH}
M\xi^*(\CPi) = \{ \sum_{i \geq 0} a_i x_\xi^i:a_i \in M\xi_* \}.
\end{equation}
\end{lem}
The filtration in the spectral sequence~\eqref{eqn:ahss}
comes from the skeleton filtration of $\CPi$ and corresponds
to powers of the augmentation ideal $\tilde{M\xi}^*(\CPi)$
in $M\xi^*(\CPi)$. Of course the product structure in the
ring $M\xi^*(\CPi)$ is more complicated than in the case
of $MU^*(\CPi)$ since $x_\xi$ is not a central element.

In order to understand a difference of the form
$u x_\xi^k - x_\xi^k u$ with $u\in M\xi_m$ and $k\geq 1$
we consider the cofibre sequence
\begin{equation*}
\Sigma^m\CP^{k-1} \subseteq \Sigma^m\CP^{k}\ra\Sigma^m\SS^{2k}.
\end{equation*}
Both elements $u x_\xi^k$ and $x_\xi^k u$ restrict to the
trivial map on $\Sigma^m\CP^{k-1}$. The orientation $x_\xi$
restricted to $\SS^2$ is the $2$-fold suspension of the unit
of $M\xi$, $\Sigma^2i\in M\xi^2(\SS^2)$. Centrality of the
unit ensures that the following square and outer diagram
\begin{equation*}
\xymatrix{
{\Sigma^m S \wedge (\Sigma^2 S)^{\wedge k}}
\ar[rrr]^{u \wedge (\Sigma^2i)^{\wedge k}} \ar[dd]_{\mathrm{twist}}
&&& {M\xi \wedge (\Sigma^2M\xi)^{\wedge k}}\ar[dd]^{\mathrm{twist}}\ar[dr] & \\
&&&& {\Sigma^{2k}M\xi} \\
{(\Sigma^2 S)^{\wedge k} \wedge \Sigma^m S}
\ar[rrr]^{(\Sigma^2i)^{\wedge k} \wedge u}
&&& {(\Sigma^2M\xi)^{\wedge k} \wedge M\xi}
\ar[ur]
}
\end{equation*}
commute, so the difference $u x_\xi^k - x_\xi^k u$ is trivial.
his yields
\begin{lem}\label{lem:difference}
For every $u\in M\xi_m$ and $k\geq 1$, $u$ and $x_\xi^k$
commute up to elements of filtration at least $2k+2$, \ie,
\begin{equation}\label{eqn:MxiCPi-commutator}
u x_\xi^k - x_\xi^k u \in (\tilde{M\xi}^*(\CPi))^{[2k+2]}.
\end{equation}
\end{lem}

Let $E$ be any associative $S$-algebra with an orientation
class $u_E \in E^2(\CPi)$. The Atiyah-Hirzebruch spectral
sequence for $E^*\CPi$ identifies $E^*(\CPi)$ with the left
$E_*$-module of power series in $u_E$ as in the case of
$M\xi$:
\begin{equation*}
E^*(\CPi) = \{ \sum_{i \geq 0} \theta_i u_E^i : \theta_i \in E_*\}.
\end{equation*}
The orientation class $u_E \in E^2(\CPi)$ restricts to the
double suspension of the unit of $E$, $\Sigma^2i_E\in E^2(\CP^1)$.
Induction on the skeleta shows that for all $n$, $E_*(\CP^n)$
is free over $E^*$ and we obtain that
\begin{equation*}
E_*(\CPi)\cong E_*\{\beta_0,\beta_1,\ldots\}
\end{equation*}
with $\beta_i \in E_{2i}(\CPi)$ being dual to $u_E^i$.
Let $\phi\:M\xi \ra H \wedge M\xi$ be the map
induced by the unit of $H$, and let
$\Theta\: M\xi^*(\CPi)\ra(H\wedge M\xi)^*(\CPi)$
be the induced ring homomorphism (in fact $\Theta$ is
a monomorphism as explained below). Then
\begin{equation}\label{eqn:Mxi-Exp}
\Theta(x_\xi)  = \sum_{i \geq 0}z_i x_H^{i+1} = z(x_H),
\end{equation}
where $x_H \in (H \wedge M\xi)^2(\CPi)$ is the orientation
coming from the canonical generator of $H^2(\CPi)$. The
proof is analogous to that for $MU$ in~\cite{Ad-1}. Note that
$H\wedge M\xi$ is an algebra spectrum over the commutative
$S$-algebra $H$ which acts centrally on $H \wedge M\xi$. Hence
$x_H$ is a central element of $(H \wedge M\xi)^*(\CPi)$. This
contrasts with the image of $x_\xi$ in $(H \wedge M\xi)^*(\CPi)$
which does not commute with all elements of $(H \wedge M\xi)_*$.

We remark that the cohomology ring $M\xi^*(\CPi)$ is highly
non-commutative. Using~\eqref{eqn:Mxi-Exp}, and noting that
coefficient $z_i\in H_{2i}(M\xi)$ is an indecomposable in
the algebra $H_*(M\xi)$, it follows that $x_\xi$ does not
commute with any of the $z_i$. For example, the first
non-trivial term in the commutator
\[
z_1 \,z(x_H) - z(x_H)\,z_1
\]
is $(z_1z_2-z_2z_1)x_H^3\neq0$.

Let $\nsym$ denote the ring of non-symmetric functions.
This ring can be identified with $H_*(\Omega\Sigma\CPi)$.
Using this and the above orientation we obtain
\begin{prop}\label{prop:MxiCPi-NSymm}
The map $\Theta$ induces a monomorphism
$\Theta\: M\xi^*(\CPi) \ra \nsym[[x_H]]$.
\end{prop}
\begin{proof}
The right-hand side is isomorphic to the
$H\wedge M\xi$-cohomology of $\CPi$. As $M\xi$ is
a wedge of suspensions of $BP$ at every prime and
as the map is also rationally injective, we obtain
the injectivity of $\Theta$.
\end{proof}

Note that for any $\lambda\in M\xi_*$ we can express
$\Theta(x_\xi\lambda)$ in the form
\begin{equation*}
\sum_{i \geq 0}z_i x_H^{i+1} \lambda
             = \sum_{i\geq 0} z_i \lambda x_H^{i+1},
\end{equation*}
but as the coefficients are non-commutative, we cannot
pass $\lambda$ to the left-hand side, so care has to
be taken when calculating in $\nsym[[x_H]]$.

\section{A formal group law over $M\xi_*$}\label{sec:FGL/Mxi}

The two evident line bundles $\eta_1,\eta_2$ over $\CPi \times \CPi$
can be tensored together to give a line bundle $\eta_1 \otimes\eta_2$
classified by a map $\mu\:\CPi \times \CPi \ra \CPi$ and by naturality
we obtain an element $\mu^*x_\xi \in M\xi^2(\CPi\times\CPi)$. We
also have
\begin{equation}\label{eqn:MxiCPiCPi-LH}
M\xi^*(\CPi \times \CPi) =
\biggl\{
     \sum_{i,j \geq 0} a_{i,j} (x'_\xi)^i(x''_\xi)^j
                                 : a_{i,j} \in M\xi_*
                                 \biggr\}
\end{equation}
as a left $M\xi^* = M\xi_{-*}$-module, where
$x'_\xi,x''_\xi \in M\xi^2(\CPi \times\CPi)$ are obtained
by pulling back $x_\xi$ along the two projections. We have
\begin{equation*}
\mu^*x_\xi = F_\xi(x'_\xi,x''_\xi) =
x'_\xi + x''_\xi + \sum_{i,j \geq 1} a_{i,j} (x'_\xi)^i(x''_\xi)^j,
\end{equation*}
where $a_{i,j} \in M\xi_{2(i+j)-2}$. The notation $F_\xi(x'_\xi,x''_\xi)$
is meant to suggest a power series, but care needs to be taken
over the use of such notation. For example, since the tensor
product of line bundles is associative up to isomorphism,
the formula
\begin{subequations}\label{eqn:Mxi-FGL}
\begin{equation}\label{eqn:Mxi-FGLassoc}
F_\xi(F_\xi(x'_\xi,x''_\xi),x'''_\xi) = F_\xi(x'_\xi,F_\xi(x''_\xi,x'''_\xi))
\end{equation}
holds in $M\xi^*(\CPi \times \CPi \times \CPi)$, where
$x'_\xi,x''_\xi,x'''_\xi$ denote the pullbacks of $x_\xi$
along the three projections. When considering this formula,
we have to bear in mind that the inserted expressions do
not commute with each other or coefficients. We also have
the identities
\begin{align}
\label{eqn:Mxi-FGL0}
F_\xi(0,x_\xi) &= x_\xi = F_\xi(0,x_\xi), \\
\label{eqn:Mxi-FGLcomm}
F_\xi(x'_\xi,x''_\xi) &= F_\xi(x''_\xi,x'_\xi).
\end{align}
Let $\bar{x}_\xi=\gamma^*x_\xi$ denotes the pullback of
$x_\xi$ along the map $\gamma\:\CPi\ra\CPi$ classifying
the inverse $\eta^{-1}=\bar{\eta}$ of the canonical line
bundle $\eta$. Then $\bar{x}_\xi\in M\xi^2(\CPi)$ and
there is a unique expansion
\begin{equation*}
\bar{x}_\xi = -x_\xi + \sum_{k\geq1} c_k x_\xi^{k+1}
\end{equation*}
with $c_k\in M\xi_{2k}$. Since $\eta\otimes\bar{\eta}$
is trivial, this gives the identities
\begin{equation*}
F_\xi(x_\xi,\bar{x}_\xi) = 0 = F_\xi(\bar{x}_\xi,x_\xi)
\end{equation*}
and so
\begin{equation}\label{eqn:Mxi-inverse}
F_\xi(x_\xi,-x_\xi + \sum_{k\geq1} c_k x_\xi^{k+1})
= 0
= F_\xi(-x_\xi + \sum_{k\geq1} c_k x_\xi^{k+1},x_\xi).
\end{equation}
\end{subequations}
To summarize, we obtain the following result.
\begin{prop}\label{prop:Mxi-commFGL}
The identities~\eqref{eqn:Mxi-FGL} together show that
$F_\xi(x'_\xi,x''_\xi)$ defines a commutative formal
group law over the non-commutative ring $M\xi_*$.
\end{prop}
\begin{rem}
Note however, that most of the classical structure theory
for formal group laws over (graded) commutative rings does
\emph{not} carry over to the general non-commutative setting.
For power series rings over associative rings where the
variable commutes with the coefficients most of the theory
works as usual. If the variable commutes with the coefficients
up to a controlled deviation, then the ring of skew power
series still behaves reasonably (see for example~\cite{D}),
but our case is more general.
\end{rem}

\section{The splitting of $M\xi$ into wedges of suspensions
of $BP$}\label{sec:Mxi-BPsplitting}

In \cite{BR:qsym} we showed that there is a splitting of $M\xi$
into a wedge of copies of suspensions of $BP$ locally at each
prime~$p$. In the case of $MU$ the inclusion of the bottom
summand is given by a map of ring spectra $BP \ra MU_{(p)}$.
However, for $M\xi$ this is not the case.
\begin{prop} \label{prop:normaps}
For each prime $p$, there is no map of ring spectra $BP\ra M\xi_{(p)}$
\end{prop}
\begin{proof}
We give the proof for an odd prime $p$, the case $p=2$ is
similar. We set $H_*=(H\F_p)_*$.

Recall that
\begin{equation*}
H_*(BP) = \F_p[t_1,t_2,\ldots]
\end{equation*}
where $t_r\in H_{2p^r-2}(BP)$ and the $\mathcal{A}_*$-coaction
on these generators is given by
\begin{equation*}
\psi(t_n) = \sum_{k=0}^n \zeta_k\otimes t_{n-k}^{p^k},
\end{equation*}
where $\zeta_r\in\mathcal{A}_{2p^r-2}$ is the conjugate of the
usual Milnor generator $\xi_r$ defined in~\cite{Ad-1}. The right
action of the Steenrod algebra satisfies
\begin{equation*}
\mathcal{P}^1_*t_1 = -1,
   \quad \mathcal{P}^1_*t_2 = -t_1^p,
   \quad \mathcal{P}^p_*t_2 = 0.
\end{equation*}
Assume that a map of ring spectra $u\colon BP\ra M\xi_{(p)}$ exists.
Then $\mathcal{P}^1_*u_*(t_1) = u_*(-1) = -1$, hence $w:= u_*(t_1)\neq 0$.
Notice that
\begin{equation*}
\mathcal{P}^1_*(w^{p+1}) = -w^p, \quad \mathcal{P}^p_*(w^{p+1})=-w.
\end{equation*}
Also, $\mathcal{P}^p_*u_*(t_2)=0$. This shows that $u_*(t_2)$ cannot
be equal to a non-zero multiple of $w^{p+1}$. Therefore it is not
contained in the polynomial subalgebra of $H_*(M\xi_{(p)})$ generated
by $w^{p+1}$ and thus it cannot commute with $w$. This shows that
the image of $u_*$ is not a commutative subalgebra of $H_*M\xi_{(p)}$
which contradicts the commutativity of $H_*BP$.
\end{proof}

\begin{rem}
Note that Proposition \ref{prop:normaps} implies that there
is no map of ring spectra from $MU$ to $M\xi$, because if
such a map existed, we could precompose it $p$-locally with
the ring map $BP \ra MU_{(p)}$ to get a map of ring spectra
$BP \ra M\xi_{(p)}$.
\end{rem}

\section{The real and the quaternionic cases}\label{sec:R&H}
\label{sec:Realcase}

Analogous to the complex case, the map $\RPi = BO(1) \ra BO$
gives rise to a loop map $\xi_\R\:\Omega\Sigma\RPi \ra BO$,
and hence there is an associated map of associative $S$-algebras
$M\xi_\mathbb{R} \ra MO$ on the level of Thom spectra. There
is a splitting of $MO$ into copies of suspensions of $H\F_2$.
In fact a stronger result holds.
\begin{prop} \label{prop:map}
There is a map of $E_2$-spectra $H\F_2 \ra MO$.
\end{prop}
\begin{proof}
The map $\alpha\:\SS^1 \ra BO$ that detects the
generator of the fundamental group of $BO$ gives rise
to a double-loop map
\begin{equation*}
\Omega^2\Sigma^2 \SS^1 = \Omega^2\SS^3 \ra BO.
\end{equation*}
As the Thom spectrum associated to $\Omega^2\SS^3$
is a model of $H\F_2$ by~\cite{M} with an $E_2$-structure \cite{LMS}, the claim follows.
\end{proof}

Generalizing an argument by Hu-Kriz-May~\cite{HKM}, Gilmour~\cite{G}
showed that there is no map of commutative $S$-algebras $H\F_2\ra MO$.

The $E_2$-structure on the map from Proposition~\ref{prop:map}
cannot be extended to $\xi_\R$. On the space level,
\begin{equation*}
H_*(\Omega\Sigma\RPi;\F_2) \cong T_{\F_2}(\bar{H}_*(\RPi;\F_2)),
\end{equation*}
where $H_n(\RPi;\F_2)$ is generated by an element $x_n$.

\begin{prop} \label{prop:noring}
There is no map of ring spectra $H\F_2 \ra M\xi_\R$.
\end{prop}
\begin{proof}
Assume $\gamma\: H\F_2 \ra M\xi_\R$ were a map of ring spectra.
We consider $\gamma_*\:(H\F_2)_*H\F_2 \ra (H\F_2)_*M\xi_\R$.
Note that $(H\F_2)_*M\xi_\R$ is the free associative $\F_2$-algebra
generated by $z_1,z_2,\ldots $ with $z_i$ in degree~$i$ being
the image of $x_i$ under the Thom-isomorphism.

Under the action of the Steenrod-algebra on $H\F_2$-homology
$\Sq_*^1(z_1) = 1$ and hence $\Sq_*^1(z_1^3) = z_1^2$ by the
derivation property of $\Sq_*^1$.

In the dual Steenrod algebra we have $\Sq_*^1(\xi_1) = 1$ and
$\Sq_*^2(\xi_2) = \xi_1$ and $\Sq_*^1(\xi_2) = 0$.

Combining these facts we obtain
\begin{equation}
\Sq_*^1(\gamma_*(\xi_1)) = \gamma_*(\Sq_*^1\xi_1) = \gamma_*(1) = 1,
\end{equation}
in particular $\gamma_*(\xi_1) \neq 0$ and thus $\gamma_*(\xi_1) = z_1$.

Similarly,
\begin{equation*}
\Sq_*^2\gamma_*(\xi_2) = \gamma_*(\Sq_*^2\xi_2) = \gamma_*\xi_1 = z_1
\neq 0.
\end{equation*}

The image of $\gamma_*$ generates a commutative sub-$\F_2$-algebra
of $(H\F_2)_*M\xi_\mathbb{R}$. The only elements in $(H\F_2)_*M\xi_\R$
that commute with $z_1$ are polynomials in $z_1$. Assume that
$\gamma_*\xi_2 = z_1^3$. Then
\begin{equation*}
0 = \gamma_*\Sq_*^1\xi_2 = \Sq_*^1(z_1^3) = z_1^2 \neq 0,
\end{equation*}
which is impossible. Therefore, $\gamma_*\xi_2$ does not
commute with $z_1$, so we get a contradiction.
\end{proof}

Note that Proposition~\ref{prop:noring} implies that there
is no loop map $\Omega^2\SS^3 \ra \Omega\Sigma \RPi$ that is compatible with the maps to $BO$
since such a map would induce a map of associative $S$-algebras
$H\F_2\ra M\xi_\R$.

A quaternionic model of quasisymmetric functions is given
by $H^*(\Omega \Sigma \HPi)$. Here, the algebraic generators
are concentrated in degrees that are divisible by~$4$. The
canonical map $\HPi = BSp(1) \ra BSp$ induces a loop-map
$\xi_\H\: \Omega \Sigma \HPi \ra BSp$ and thus gives rise
to a map of associative $S$-algebras on the level of Thom
spectra $M\xi_\H \ra MSp$.

Of course, the spectrum $MSp$ is not as well understood as
$MO$ and $MU$. There is a commutative $S$-algebra structure
on $MSp$~\cite[pp.~22,~76]{May}, but for instance the homotopy
groups of $MSp$ are not known in an explicit form.

\section{Associative versus commutative orientations}\label{sec:Assoc-vs-Comm}

We work with the second desuspension of the suspension spectrum
of $\CPi$. Such spectra are inclusion prespectra~\cite[X.4.1]{EKMM}
and thus a map of $S$-modules from $S = \Sigma^\infty\SS^0$ to
$\Sigma^{\infty -2}\CPi:=\Sigma^{-2}\Sigma^\infty \CPi$ is given
by a map from the zeroth space of the sphere spectrum to the zeroth
space of $\Sigma^{\infty -2}\CPi$ which in turn is a colimit, namely
$\colim_{\R^2 \subset W} \Omega^W \Sigma^{W-\R^2}\CPi$. As a map
$\varrho\: S \ra \Sigma^{\infty -2} \CPi$ we take the one that is
induced by the inclusion $\SS^2 = \CP^1\subset\CPi$.


The commutative $S$-algebra $S\wedge (\N_0)_+ = S[\N_0]$ has
a canonical map $ S[\N_0] \ra S$ which is given by the fold
map. We can model this via the map of monoids that sends the
additive monoid $(\N_0,0,+)$ to the monoid $(0,0,+)$; thus
$S[\N_0] \ra S$ is a map of commutative $S$-algebras.

We get a map $S[\N_0] \ra \A(\Sigma^{\infty-2}\CPi)$ by taking
the following map on the $n$th copy of $S$ in $S[\N_0]$. We can
view $S$ as $S \wedge \{*\}_+$ where $\{*\}$ is a one-point space.
The $n$-fold space diagonal gives a map
\begin{equation*}
\delta_n\: S = S\wedge \{*\}_+ \ra
   S\wedge \{(\underbrace{*,\ldots,*)}_{n}\}_+
      \simeq S^{\wedge n}
\end{equation*}
which fixes an equivalence of $S$ with $S^{\wedge n}$. We compose
this map with the $n$-fold smash product of the map
$\varrho\: S \ra \Sigma^{\infty-2}\CPi$. The maps
\begin{equation*}
\varrho^{\wedge n} \circ \delta_n\: S \ra
(\Sigma^{\infty-2}\CPi)^{\wedge n} \ra \A(\Sigma^{\infty-2}\CPi)
\end{equation*}
together give a map of $S$-algebras
\begin{equation*}
\tau\: S[\N_0] \ra \A(\Sigma^{\infty-2}\CPi).
\end{equation*}

Note, however, that $S[\N_0]$ is not central in $\A(\Sigma^{\infty-2}\CPi)$.
Thus the coequalizer
\begin{equation*}
S \wedge_{S[\N_0]} \A(\Sigma^{\infty-2}\CPi)
\end{equation*}
does not possess any obvious $S$-algebra structure. Furthermore,
there is a natural map
\begin{equation*}
S \wedge_{S[\N_0]} \A(\Sigma^{\infty-2}\CPi) \ra M\xi,
\end{equation*}
but this is not a weak equivalence since the $H\Z$-homology of
the left-hand side is the quotient by the left ideal generated
by $z_0-1$ and thus it is bigger than $H\Z_*M\xi$ which is the
quotient by the two-sided ideal generated by $z_0-1$.

In the commutative context the pushout of commutative $S$-algebras
is given by the smash product. Hence there is a natural morphism
of commutative $S$-algebras
\begin{equation*}
\tilde{\P}(\Sigma^{\infty-2}\CPi) =
 S \wedge_{\P(\Sc)} \P(\Sigma^{\infty-2}\CPi) \ra MU,
\end{equation*}
where $\tilde{\P}(\Sigma^{\infty-2}\CPi)$ is the pushout
in the following diagram of commutative $S$-algebras
\begin{equation*}
\xymatrix{
\P(\Sc) \ar[r]  \ar[d]
\ar@{}[dr]|<<<<<<<<{\text{\Large$\ulcorner$}}
       & \P(\Sigma^{\infty-2}\CPi) \ar[d] \\
S \ar[r] & \tilde{\P}(\Sigma^{\infty-2}\CPi).
}
\end{equation*}
Here, we use the identity map on $S$ to induce the left-hand
vertical map of commutative $S$-algebras and the inclusion
of the bottom cell of $\Sigma^{\infty-2}\CPi$ to induce the
top map which is a cofibration and therefore
$\tilde{\P}(\Sigma^{\infty-2}\CPi)$ is cofibrant. However,
the map $\tilde{\P}(\Sigma^{\infty-2}\CPi) \ra MU$ is not
a weak equivalence as the next result shows.

\begin{lem}\label{lem:rat-eqce-MU}
The canonical map of commutative $S$-algebras
\begin{equation*}
\tilde{\P}(\Sigma^{\infty-2}\CPi) \ra MU
\end{equation*}
is an equivalence rationally, but not globally. Furthermore,
there is a morphism of ring spectra
\begin{equation*}
MU \ra \tilde{\P}(\Sigma^{\infty-2}\CPi)
\end{equation*}
which turns $MU$ into a retract of\/ $\tilde{\P}(\Sigma^{\infty-2}\CPi)$.
\end{lem}
\begin{proof}
Let $\k$ be a field. The K\"unneth spectral sequence for
the homotopy groups of
\begin{equation*}
H\k \wedge (\tilde{\P}(\Sigma^{\infty-2}\CPi))
\simeq
  H\k\wedge_{\P_{H\k}(H\k)} \P_{Hk}(\Sigma^{-2}H\k\wedge\CPi)
\end{equation*}
has $\mathrm{E}^2$-term
\begin{equation*}
\mathrm{E}^2_{*,*} =
\Tor_{*,*}^{\pi_*(\P_{H\k}(H\k))}(\k,\pi_*(\P_{H\k}(\Sigma^{-2}H\k\wedge\CPi))).
\end{equation*}
When $\k = \Q$, $\pi_*(\P_{H\Q}(H\Q))$ is a polynomial algebra
on a zero-dimensional class $x_0$ and
\begin{equation} \label{eq:htp-free}
\pi_*(\P_{H\Q}(\Sigma^{-2}H\Q\wedge\CPi)) \cong \Q[x_0,x_1,\ldots],
\end{equation}
where $|x_i|=2i$. Thus
\begin{equation*}
\pi_*(H\Q \wedge (\widetilde{\P}(\Sigma^{\infty-2}\CPi))
  \cong \Q[x_1,x_2,\ldots] \cong H\Q_*(MU).
\end{equation*}

However, when $\k=\F_p$ for a prime $p$, the freeness of
the commutative $S$-algebras $\P(S)$ and
$\P(\Sigma^{\infty-2}\CPi))$ implies that
$(H\F_p)_*(\P(\Sigma^{\infty-2}\CPi))$ is a free
$(H\F_p)_*(\P(S))$-module and thus the $\mathrm{E}^2$-term
reduces to the tensor product in homological degree zero.
Note that this tensor product contains elements of odd
degree, but $(H\F_p)_*(MU)$ doesn't.

Using the orientation for line bundles given by the canonical
inclusion
\begin{equation*}
\Sigma^{\infty-2}\CPi\ra \widetilde{\P}(\Sigma^{\infty-2}\CPi),
\end{equation*}
we have a map of ring spectra
\begin{equation*}
\phi \: MU \ra \widetilde{\P}(\Sigma^{\infty-2}\CPi).
\end{equation*}

The inclusion map $\CPi = BU(1) \ra BU$ gives rise to the
canonical map $\sigma\: \Sigma^{\infty-2}\CPi \ra MU$
and with this orientation we get a morphism of commutative
$S$-algebras
\begin{equation*}
\theta\: \widetilde{\P}(\Sigma^{\infty-2}\CPi)\ra MU,
\end{equation*}
such that the composite $\theta\circ\phi\circ\sigma$ agrees
with $\sigma$, hence $\theta\circ\varphi$ is homotopic to
the identity on $MU$.
\end{proof}

Using topological Andr\'e-Quillen homology, $\TAQ_*(-)$,
we can show that the map of ring spectra $\phi$ cannot
be rigidified to a map $\tilde{\phi}$ of commutative
$S$-algebras in such a way that the composite
$\theta\circ\tilde{\phi}$ is a weak-equivalence. By
Basterra-Mandell~\cite{BM},
\begin{equation*}
\TAQ_*(MU|S;H\F_p) \cong (H\F_p)_*(\Sigma^2 ku),
\end{equation*}
while \cite[proposition~1.6]{BGR} together with subsequent
work of the first named author
\cite{Ba:TAQI}
gives
\begin{equation*}
\TAQ_*(\widetilde{\P}(\Sigma^{\infty-2}\CPi)|S;H\F_p)
    \cong (H\F_p)_*(\Sigma^{\infty-2}\CPi_2),
\end{equation*}
where $\CPi_2 = \CPi/\CP^1$ is the cofiber of the inclusion
of the bottom cell.
\begin{prop}\label{prop:MU->PCP}
For a prime $p$, there can be no morphism of commutative
$S_{(p)}$-algebras
\begin{equation*}
\theta_{(p)}\:MU_{(p)}\ra(\tilde{\mathbb{P}}\Sigma^{\infty-2}\CPi)_{(p)}
\end{equation*}
for which $\sigma_{(p)}\circ\theta_{(p)}$ is a weak equivalence.
Hence there can be no morphism of commutative $S$-algebras
\begin{equation*}
\theta\:MU\lra\tilde{\mathbb{P}}\Sigma^{\infty-2}\CPi
\end{equation*}
for which $\sigma\circ\theta$ is a weak equivalence.
\end{prop}
\begin{proof}
It suffices to prove the first result, and we will assume that
all spectra are localised at~$p$. Assume such a morphism $\theta$
existed. Then by naturality of the functor of K\"ahler differentials,
$\Omega_S$, there are (derived) morphisms of $MU$-modules and
a commutative diagram
\begin{equation*}
\xymatrix{
\Omega_{S}(MU) \ar[r]_(.4){\theta_*} \ar@/^19pt/[rr]^{\sim}
& \Omega_{S}(\tilde{\mathbb{P}}\Sigma^{\infty-2}\CPi)\ar[r]_(.6){\sigma_*}
& \Omega_{S}(MU)
}
\end{equation*}
which by~\cite{BM} 
induce a commutative diagram in $\TAQ_*(-;H\F_p)$
of the following form:
\begin{equation*}
\xymatrix{
H_*(\Sigma^2ku;\F_p) \ar[r]_(.43){\theta_*} \ar@/^19pt/[rr]^{\iso}
& H_*(\Sigma^{\infty-2}\CPi_2;\F_p)\ar[r]_(.56){\sigma_*}
& H_*(\Sigma^2ku;\F_p).
}
\end{equation*}
It is standard that
\begin{equation*}
H_n(\Sigma^{\infty-2}\CPi_2;\F_p) =
\begin{cases}
\F_p & \text{if $n\geq2$ and is even}, \\
\;0 & \text{otherwise}.
\end{cases}
\end{equation*}
On the other hand, when $p=2$,
\begin{equation*}
H_*(ku;\F_2) =
\F_2[\zeta_1^2,\zeta_2^2,\zeta_3,\zeta_4,\ldots]
                           \subseteq\mathcal{A}(2)_*
\end{equation*}
with $|\zeta_s|=2^s-1$, while when $p$ is odd
\begin{equation*}
\Sigma^2ku_{(p)} \sim \bigvee_{1\leq r\leq p-1}\Sigma^{2r}\ell,
\end{equation*}
where $\ell$ is the Adams summand with
\begin{equation*}
H_*(\ell;\F_2) =
\F_p[\zeta_1,\zeta_2,\zeta_3,\ldots]\otimes\Lambda(\bar{\tau}_r:r\geq2),
\end{equation*}
for $|\zeta_s|=2p^s-2$ and $|\bar{\tau}_s|=2p^s-1$.
Hence no such $\theta$ can exist.
\end{proof}
\begin{prop}\label{prop:MU-notuniv}
There are commutative $S$-algebras $E$ which possess
a map of commutative $S$-algebras
\begin{equation*}
\tilde{\P}(\Sigma^{\infty-2}\CPi) \ra E
\end{equation*}
that cannot be extended to a map of commutative $S$-algebras
$MU\ra E$.
\end{prop}
\begin{proof}
Matthew Ando~\cite{ando} constructed complex orientations
for the Lubin-Tate spectra $E_n$ which are $H_\infty$-maps
$MU\ra E_n$. However, in~\cite{JN}, Niles Johnson and Justin
No\"el showed that none of these are $p$-typical for all
primes up to at least~$13$ (and subsequently verified for
primes up to $61$). For any $p$-typical orientation there
is a map of ring spectra $MU \ra E_n$, but this map cannot
be an $H_\infty$-map and therefore is not a map of commutative
$S$-algebras.
\end{proof}

\end{document}